\documentclass[12pt]{article}
\usepackage[utf8]{inputenc}
\usepackage[english]{babel}
\usepackage{amsmath, amsfonts, amssymb, amsthm, amscd}
\usepackage{enumitem}
\usepackage{xcolor}
\usepackage{url}
\usepackage{cleveref}

\newtheorem{theorem}{Theorem}[section]
\newtheorem{fact}[theorem]{Fact}
\newtheorem{lemma}[theorem]{Lemma}
\newtheorem{proposition}[theorem]{Proposition}
\newtheorem*{proposition*}{Proposition}
\newtheorem*{theorem*}{Theorem}

\theoremstyle{definition}
\newtheorem{definition}[theorem]{Definition}

\DeclareMathOperator{\tp}{tp}

\title{A note on Pontryagin duality and continuous logic}
\date{\today}
\author{Nicolas Chavarria\\University of Notre Dame \and Anand Pillay\thanks{Supported by NSF grants DMS-1665035, DMS-1760212, and DMS-2054271}\\University of Notre Dame}

\begin{document}
  
    \maketitle

    \begin{abstract}
        We exhibit Pontryagin duality as a special case of Stone duality in a continuous logic setting. More specifically, given an abelian topological group $A$, and $\mathcal F$ the family (group) of continuous homomorphisms from $A$ to the circle group $\mathbb T$, then, viewing $(A,+)$ equipped with the collection $\mathcal F$ as a continuous logic structure $M$, we show that the local type space $S_\mathcal F(M)$ is precisely the Pontryagin dual of the group $\mathcal F$ where the latter is considered as a discrete group.
        
        We conclude, using Pontryagin duality (between compact and discrete abelian groups), that $S_\mathcal F(M)$ is the Bohr compactification of the topological group $A$. 
    \end{abstract}

    \section{Introduction}

    We make some observations relating type spaces in continuous logic, Pontryagin duals, and Bohr compactifications. Writing this down was motivated by some questions of Marcus Tressl about our paper \cite{chavarriaPillay} asking what, if anything, is the connection with Pontryagin duality.

    The point of our paper \cite{chavarriaPillay} was that the theory, in continuous logic, of structures of the form ---an abelian group equipped with homomorphisms to a compact group--- is stable. However, stability is not relevant to the issues we discuss in the present paper, which will be elementary at the technical level, but possibly of general interest.

    Given an abelian topological group $A$, we consider the family $\mathcal F$ of all \emph{continuous} homomorphisms to $\mathbb T$, and consider $(A,+,-,0,f)_{f\in {\mathcal F}}$ as a continuous logic structure $M$. We form the space $S_{\mathcal F}(M)$ of ${\mathcal F}$-types over $M$, a compact space. There are two points that we make in this paper:
    \begin{enumerate}[label=(\roman*)]
        \item $S_\mathcal F(M)$, with its natural group structure, is the Pontryagin dual of ${\mathcal F}$ where the latter is considered as a discrete abelian group, and
        \item $S_{\mathcal F}(M)$, with the natural homomorphism from $A$, is the Bohr compactification $\iota:A\to bA$ of $A$.
    \end{enumerate}
    Point (i) says that the move from the family ${\mathcal F}$ of ``formulas" to the space of ${\mathcal F}$-types is precisely taking the Pontryagin dual. Recall that the Pontryagin dual of a discrete abelian group $B$ is the group of all homomorphisms from $B$ to $\mathbb T$, with the Tychonoff topology (induced from ${\mathbb T}^{B}$). And the Bohr compactification of a topological group is the universal object among continuous homomorphisms to compact groups. Then point (ii) follows from (i) by Pontryagin duality. Both (i) and (ii) are explained (proved) in \Cref{section:mainResult}.

Thanks to the referee for the helpful suggestions and comments.

    \section{Preliminaries}

    We go through the necessary background. When we talk about topological spaces or groups, we will always assume that the topology is Hausdorff. Throughout, $\mathbb T$ will denote the circle group, a compact abelian topological group which also has a compatible metric. We will treat it as a subset of $\mathbb C$ and so write it multiplicatively.

    \subsection{Bohr compactification and Pontryagin duality}

    We repeat the definition of Bohr compactifications. Given an arbitrary topological group $G$, we can consider continuous homomorphisms $f$ from $G$ to compact topological groups $C$. There is a universal such $f:G\to C$, namely such that for every continuous homomorphism $g$ from $G$ to a compact group $D$ there is unique continuous homomorphism $h:C\to D$ such that $g=h\circ f$. This universal object is called the Bohr compactification of $G$, and sometimes written $bG$, or $\iota:G\to bG$ to include the continuous homomorphism $\iota$. 
    
    \begin{definition}
        Let $A$ be an abelian topological group. By $\widehat A$, the Pontryagin dual of $A$, we mean the group of continuous homomorphisms from $A$ to $\mathbb T$ equipped with the compact-open topology.
    \end{definition}

    Recall that this compact-open topology has as a sub-basis of opens, $\{f\in\widehat A:f(K)\subseteq U\}$, where $K\subseteq A$ is compact, and $U\subseteq {\mathbb T}$ is open. It coincides with the topology of uniform convergence on compact subsets of $A$ (using the metric on $\mathbb T$). When $A$ is discrete the topology is precisely that induced by the Tychonoff topology on $\mathbb T^A$. 

    Given an abelian topological group $A$, we have a homomorphism $ev$ (for evaluation) from $A$ to $\widehat{\widehat A}$, where $ev(a)$ is the continuous function from $\widehat A$ to $\mathbb T$ which takes $f$ to $f(a)$. Pontryagin duality says that, for $A$ a {\em locally compact} abelian group $A$, this homomorphism is an isomorphism of topological groups. We will only use it in the case of compact $A$, which we state now (see \cite{katznelson}, Chapter IV). 
    
    \begin{fact}\label{fact:pontryagin}
        \begin{enumerate}[label=(\roman*)]
            \item If $A$ is compact abelian, then $\widehat A$ is discrete abelian, and if $A$ is discrete then $\widehat A$ is compact.
            \item For $A$ compact, $ev:A\to \widehat{\widehat A}$ is an isomorphism of topological groups.
        \end{enumerate}
    \end{fact}

    We can now state a (well-known) description in the above terms of the Bohr compactification of an abelian topological group. This is often stated in the literature only for locally compact abelian groups (and sometimes even as the definition). See \cite{katznelson}, VII.5. We give a proof for completeness. 

    \begin{lemma}\label{lemma:bohrCompactification}
        Let $A$ be an abelian topological group. Let $(\widehat A)_d$ be $\widehat A$ but with the discrete topology. Then $bA=\widehat{(\widehat A)_d}$, namely $A\to bA$ is isomorphic to the natural evaluation map $ev: A\to\widehat{(\widehat A)_{d}}$
    \end{lemma}

    \begin{proof}
        Let $\iota$ be the canonical continuous homomorphism from $A$ to $bA$, which has dense image.
        
        \hfill
        
        \noindent {\em Claim.} $(\widehat A)_d$ is (naturally) isomorphic to the group $\widehat{bA}$ (which note is discrete by \Cref{fact:pontryagin} (i)).

        \noindent {\em Proof of Claim.} For each continuous homomorphism $f$ from $bA$ to $\mathbb T$, $f\circ\iota$ is a continuous homomorphism from $A$ to $\mathbb T$. Moreover, by the universality property of $bA$, the homomorphism taking $f$ to $f\circ\iota$ establishes an isomorphism (of abstract groups) between the group of homomorphisms from $A$ to $\mathbb T$, and the group of homomorphisms from $bA$ to $\mathbb T$. But as $bA$ is compact, $\widehat{bA}$ is discrete, hence the map taking $f$ to $f\circ \iota$ is an isomorphism $i$ of topological groups between $(\widehat A)_d$ and $\widehat{bA}$.\qed

        By the claim, $\widehat{(\widehat A)_d}$ is isomorphic as a topological group to $\widehat{(\widehat {bA})}$, which by \Cref{fact:pontryagin} is equal to $bA$. Checking the appropriate maps gives the full statement of the proposition. 
    \end{proof}

    \subsection{Continuous logic}

    We discuss briefly continuous logic as in \cite{benyaacovBerensteinHensonUsvyatsov}, where passing from formulas to type spaces can also be interpreted as passing from a C\textsuperscript*-algebra to its Gelfand space (as we mention below). A different version of continuous logic is developed in \cite{chavarriaPillay}, from whence the question regarding Pontryagin duality originally sprang. Modulo some care regarding the definition and topologization of the space of types, the result and proof in \Cref{section:mainResult} go through essentially unchanged. We refer to the original paper for the interested reader.

    We first remind the reader of ``local type spaces" in classical first order logic. We consider an $L$-structure $M$, and a collection ${\mathcal F}$ of $L$-formulas of the form $\phi(x,\overline y)$ where the single variable $x$ is fixed and $\overline y$ can vary depending on $\phi$. Let $a\in M^*$, a (saturated) elementary extension of $M$. The $\mathcal F$-type of $a$ over $M$ is the information of which formulas of the form $\phi(x,\overline b)$ for $\phi(x,\overline y)$ in ${\mathcal F}$ and $\overline b$ a tuple from $M$ are true of $a$ in $M^*$. We call this collection of $\mathcal F$-types $S_{\mathcal F}(M)$. The formulas $\phi(x,\overline b)$ for $\phi\in {\mathcal F}$ and $\overline b$ in $M$ (modulo equivalence in $M$) generate a Boolean algebra (under $\wedge, \vee, \neg$) and $S_{\mathcal F}(M)$ is precisely the set of ultrafiliters on this Boolean algebra, equipped naturally with the Stone space topology. So the passage from a collection ${\mathcal F}$ of $ L$-formulas with a distinguished free variable $x$ to the collection of ${\mathcal F}$-types over $M$ is precisely passing from a Boolean algebra to its Stone space. 

    Continuous logic is about logic where the formulas are real- or, in our case, complex-valued and the structures may be equipped with a metric (which is equality in the discrete case). It has been developed in a mathematical logic or model-theoretic framework since the 1960's, starting with Chang and Keisler's book \cite{changKeisler}. Since about 15 years or so, a certain attractive framework has been developed, the basics of which are expounded in the papers \cite{benyaacovBerensteinHensonUsvyatsov} and \cite{benyaacovUsvyatsov}, and which the reader is referred to. Some of the interest or motivation is to do model theory or classification theory in more general environments, as well as develop tools to apply model theory to functional analysis (and even combinatorics). There are various informal and formal connections with many-valued logics which other logic communities have pursued for a long time. 

    We are interested in abelian groups $(A,+,-,0)$ with a collection ${\mathcal F}$ of homomorphisms from $A$ to ${\mathbb T}$. From the point of view of continuous logic, we can view these homomorphism as $\mathbb C$-valued predicates. We think of $+$, $-$, $0$ as (binary, unary, zero-ary) functions (or function symbols) on $A$.  Treat $=$ as a $\{0,1\}$-valued formula with value $0$ meaning equal. Atomic formulas are of the form $f(t)$ where $f\in\mathcal F$ and $t$ a term. The ``propositional connectives" are continuous functions from $\mathbb C^n$ to $\mathbb C$. The quantifiers are $\sup_x|-|$ and $\inf_x|-|$ (which are now $\mathbb R$-valued). From these we build up the relevant formulas of continuous logic. We denote by $M$ this continuous logic structure $(A,+,-,0, f)_{f\in {\mathcal F}}$ of $(A,+,-,0)$ equipped with the $\mathbb C$-valued predicates for functions in ${\mathcal F}$.

    The compactness theorem is valid, and we have saturated elementary extensions $M^*=(A^*,+^*,-^*,0^*,f^*)_{f\in\mathcal F}$ of $M$. The $\mathbb T$-valued formulas $f\in {\mathcal F}$ are without additional parameter variables. The ${\mathcal F}$-type of an element $a\in A^{*}$, $\tp_\mathcal F(a)$, is the information consisting of the values $f^{*}(a)\in {\mathbb T}$ for $f\in {\mathcal F}$ (where $f^{*}$ is the interpretation of the predicate symbol for $f$ in the model $M^{*}$). We might want to say $\tp_\mathcal F(a/M)$ but there are no parameters from $M$ to worry about. We call this collection of types $S_{\mathcal F}(M)$ (or $S_{\mathcal F}(T)$ where $T$ is the continuous logic theory of $M$). The topology on $S_{\mathcal F}(M)$ is given as follows: a basic closed set is the set of types containing the information $u(f_{1}(x),..,f_{n}(x))=r$ for $f_{i}\in {\mathcal F}$, $u:{\mathbb C}^{n}\to {\mathbb C}$ continuous and $r\in\mathbb C$. This is a compact Hausdorff space but not necessarily profinite.

    We can also see all this as the ``usual" adaptation of Stone duality to not necessarily profinite spaces. Namely, we can consider the C\textsuperscript*-algebra generated by the family ${\mathcal F}$ of continuous functions from $A$ to ${\mathbb T}\subseteq\mathbb C$, and then the Gelfand space of this C\textsuperscript*-algebra will be, as a topological space, precisely the space $S_{\mathcal F}(M)$. So from this point of view our results will exhibit Pontryagin duality as a special case of Gelfand duality. 

    Note that the quantifiers play no role in the type space $S_{\mathcal F}(M)$ (other than in defining the elementary extension $M^{*}$ of $M$ where we realize these types.)

    \section{Theorem and Proof}\label{section:mainResult}

    Let us set up the context for the main theorem. We fix an abelian topological group $A$, and let $\mathcal F$ be the collection of all {\em continuous} homomorphisms from $A$ to $\mathbb T$. So $\mathcal F$ is a group isomorphic to $\widehat A$, where the latter is considered just as an abstract group. We let $M$ be the continuous logic structure $(A,+,-,0,f)_{f\in\mathcal F}$ as in the last section. And let $M^*$ be a saturated elementary extension.

    Consider the type space $S_{\mathcal F}(M)$ as defined earlier. Note that, for $p\in S_\mathcal F(M)$, $f^*(p)$ makes sense defined as $f^{*}(a)\in {\mathbb T}$ for some/any $a$ realizing $p$. The topology on $S_\mathcal F(M)$ is the coarsest making all of these maps continuous. Now, for $p,q\in S_{\mathcal F}(M)$, define $p+^{*}q=\tp_{\mathcal F}(a+^{*}b)$ where $a$ realizes $p$ and $b$ realizes $q$. Note that $+^{*}$ is well-defined, because if $a$ realizes $p$, $b$ realizes $q$, and $f\in\mathcal F$ then the value of $f^{*}$ at $a+^{*} b$ is $f^*(a)f^*(b)$, so it depends only on the value of $f^*$ at $a$ and at $b$, so it depends only on $p$ and $q$.

    \begin{theorem}\label{theorem:mainResult}
        \begin{enumerate}[label=(\roman*)]
            \item The operation $+^*$, together with the type space topology on $S_\mathcal F(M)$, makes $S_\mathcal F(M)$ into a compact abelian topological group.
            \item This compact abelian topological group structure on $S_\mathcal F(M)$ is precisely the Pontryagin dual of $\mathcal F$ (where $\mathcal F$ is considered as a discrete group), under the identification of $p\in S_\mathcal F(M)$ with the map taking $f\in\mathcal F$ to $f^*(p)$.
            \item The map $j:A\to S_\mathcal F(M)$ taking $a\in A$ to $\tp_\mathcal F(a)$, is a continuous homomorphism which is moreover (isomorphic to) the Bohr compactification $\iota:A \to bA$ of $A$.
        \end{enumerate}
    \end{theorem}

    \begin{proof}
        (i) The fact that $+^*$ gives a group operation on $S_\mathcal F(M)$ is obvious from the definition. To show that it is a topological group, fix some $q_0,q_1\in S_\mathcal F(M)$ and let $U=\{p\in S_\mathcal F(M):|f^*(p)-t_0t_1|<\epsilon\}$ for some $f\in\mathcal F$ and $\epsilon>0$, where $t_0=f^*(q_0)$ and $t_1=(f^*(q_1))^{-1}$. This is a subbasic open neighbourhood of $q_0-^*q_1$ in $S_\mathcal F(M)$. Let $V_0=\{p\in S_\mathcal F(M):|f^*(p)-t_0|<\epsilon/2\}$ and $V_1=\{p\in S_\mathcal F(M):|(f^{-1})^*(p)-t_1|<\epsilon/2\}$. Take $p_0\in V_0$, $p_1\in V_1$ and let $b_0,b_1\in A^*$ realize $p_0$ and $p_1$, respectively. Then
        \[
        |f^*(b_0-b_1)-t_0t_1|\leq|f^*(b_0)||(f^{-1})^*(b_1)-t_1|+|t_1||f^*(b_0)-t_0|<\epsilon. 
        \]
        (Remember that $f^*(b_0),t_1\in\mathbb T$) We conclude that the map $(p,q)\mapsto p-^*q$ is continuous at $(q_0,q_1)$ for all such pairs, as desired.

        (ii) Each $p\in S_{\mathcal F}(M)$ gives rise to a homomorphism from the group $\mathcal F$ to $\mathbb T$, which we call $p^{*}$ and which is precisely the evaluation of $f^{*}$ at a realization of $p$: $p^{*}(f) = f^{*}(a)\in {\mathbb T}$, where $a$ realizes $p$ and $f\in {\mathcal F}$. This clearly gives a homomorphism $h$ from $S_{\mathcal F}(M)$ to the Pontryagin dual $\widehat{\mathcal F}$ of the (discrete abelian) group ${\mathcal F}$. $h$ is also an embedding because $p^{*} = q^{*}$ implies $f^{*}(a) = f^{*}(b)$ for all $f\in {\mathcal F}$ and realizations $a$ of $p$ and $q$ of $p$.
        
        We want $h$ to be continuous. $\widehat{\mathcal F}$ has the compact-open topology, so take a basic open given by $O=\{\alpha:\alpha(K)\subseteq U\}$, where $K$ is a compact so finite subset $\{f_{1},\ldots,f_{n}\}$ of $\mathcal F$ and $U$ is an open subset of $\mathbb T$. Then $p^*\in O$ iff $(f_{1}(x),\ldots,f_{n}(x))\in U^n$ is in $p$, which is an open condition on $p$ in $S_{\mathcal F}(M)$. 

        So we have that $h$ is a continuous embedding of the compact group $S_{\mathcal F}(M)$ in the compact group $\widehat{\widehat{\mathcal F}}$. So $h(S_{\mathcal F}(M))$ is a closed subgroup of the image. However $h(S_{\mathcal F}(M))$ separates points in $\widehat{\mathcal F}$. To see this, suppose $f_{1}\neq f_{2}\in {\mathcal F}$. Suppose $a\in A$ and $f_{1}(a) \neq f_{2}(a)$. Let $p=\tp_{\mathcal F}(a)$. Then $p^{*}(f_{1}) \neq p^{*}(f_{2})$. Hence, by Stone-Weierstrass, $h(S_\mathcal F(M))=\widehat{\mathcal F}$. 

        (iii) $j$ is a homomorphism by definition of the group structure on $S_{\mathcal F}(M)$. For continuity: Let $C$ be a closed subset of $S_{\mathcal F}(M)$, so defined by a formula $(f_{1}(x),...,f_{n}(x))\in D$ for some closed $D\subseteq {\mathbb T}^{n}$. Then $j^{-1}(C) = \{a\in A: (f_{1}(a),..,f_{n}(a))\in D\}$ which is closed in the topological group $A$ because the $f_{i}:A\to {\mathbb T}$ are continuous. 
    
        Let $\iota: A\to bA$ be the Bohr compactification of $A$. By \Cref{lemma:bohrCompactification} this is isomorphic to $ev:A\to\widehat{({\hat A})_{d}}$. But $(\widehat A)_d$ is precisely the discrete group ${\mathcal F}$. So by part (ii) (and inspection of the maps), $ev:A\to\widehat{(\widehat A)_d}$ identifies with $j:A\to S_{\mathcal F}(M)$. This gives the result.
    \end{proof}

    \section{Comparison with the literature}
In this section we comment on the extent to which Theorem \ref{theorem:mainResult} can be extracted from or is related to results in the literature.  In Subsection 4.1 we make the connection with \cite{krupinskiPillay}. In Subsection 4.2 we make the connection with results appearing in \cite{folland} including the relation between Pontryagin and Gelfand dualities. 

    \subsection{Bohr and Stone-{\v C}ech compactifications}
    
    There are a few papers dealing with the relation between the Stone-{\v C}ech compactification $\beta G$ of a topological group $G$ and the Bohr compactification $bG$ of $G$. See, for example, \cite{zlatos} for the case of discrete abelian $G$. But we should mention to begin with that it is immediate that $bG$ is a quotient of $\beta G$. This is because $\beta G$ is the universal compact space on which $G$ acts by homeomorphisms with a dense orbit, and $G$ also acts by homeomorphisms on $bG$ with dense orbit. 

    In any case, in \cite{krupinskiPillay} we also describe $bG$ for $G$ a topological group as an image of $\beta G$, working in classical first order logic/nonstandard analysis, and this will give another account of Theorem 3.1 (iii) which we briefly describe.

    We fix a topological group $G$ and let $M$ be the structure consisting of $G$, its group operation, and predicates for {\em all} subsets of $G$. The space of complete $1$-types over $M$, which we write as $S_{1}(M)$ is by definition the space of ultrafilters on the Boolean algebra of all subsets of $G$, which is the same thing as the Stone-\v{C}ech compactification $\beta G$ of $G$. Let $M^{*}$ be a sufficiently saturated elementary extension of $M$ (nonstandard model), and write $M^{*} = (G^{*}, \times^{*}, P^{*})_{P\subseteq G}$. Then there is a smallest subgroup of $G^{*}$ which has index at most $2^{2^{|G|}}$ and is the intersection of some subsets of $G$ of the form $U^{*}$ where $U$ is an {\em open} subset of $G$. We call this subgroup $(G^{*})^{00}_{top}$. It is proved in Fact 2.4 of \cite{krupinskiPillay} that $(G^{*})^{00}_{top}$ is a normal subgroup of $G^{*}$ and that the quotient group $G^{*}/(G^{*})^{00}_{top}$ with the ``logic topology" and the natural homomorphism from $G$ is the Bohr compactification $bG$ of the topological group $G$. (The statement actually first appears in \cite{gismatullinPenazziPillay}.) 

    The logic topology is as follows. First, because of the ``bounded index" condition, the quotient homomorphism $G^{*}\to G^{*}/(G^{*})^{00}_{top}$ factors through $S_{1}(M) = \beta G$, so we have a surjective map $\pi:\beta G \to G^{*}/(G^{*})^{00}_{top}$ and the logic topology is the quotient topology. In particular, $bG$ is a continuous image of $\beta G$. 

    The proof of Lemma 2.2 of \cite{krupinskiPillay} yields that any map $f$ from $G$ to a compact space $C$ extends uniquely to a map $f^{*}:G^{*}\to C$ with the property that $f^{*}$ factors through a continuous map $S_{1}(M)\to C$. Moreover, if $C$ has a compact group structure and $f$ is a homomorphism, then $f^*$ is also a homomorphism.

    The only ``new thing" we say here is the following:

    \begin{proposition}
        Assume now that $G = A$ is an abelian topological group. Then $(A^{*})^{00}_{top}=\bigcap_{f\in\widehat A}\ker(f^*)$.
    \end{proposition}

    \begin{proof}
        This comes out of the proof of Fact 2.4 (ii) in \cite{krupinskiPillay}. Let $f\in\widehat A$. By Lemma 2.2 of \cite{krupinskiPillay}, for any $f\in\widehat A$, $\ker(f^*)$ is an intersection of sets $U^*$, where $U$ is an open subset of $A$, whereby $(A^*)^{00}_{top}\subseteq\bigcap_{f\in\widehat A}\ker(f^*)$.

        Hence, we have a surjective continuous homomorphism from $A^*/(A^*)^{00}_{top}$ to $A^{*}/\bigcap_{f\in\widehat A}\ker(f^*)$. We claim that this is an isomorphism (of topological groups). If not, there is nonzero $x\in A^*/(A^*)^{00}_{top}$ whose image in $A^*/\bigcap_{f\in\widehat A}\ker(f^*)$ is $0$. But, from the structure of compact abelian groups, there will be a continuous homomorphism $h:A^{*}/(A^{*})^{00}_{top}\to\mathbb T$ such that $h(x)\neq 0$. Considering the composition of the canonical homomorphism from $A^*$ to $A^*/(A^*)^{00}_{top}$ with $h$, we obtain a (suitable) homomorphism from $A^*$ to $\mathbb T$ which is nonzero, contradicting that the image of $x$ in $A^*/\bigcap_{f\in\widehat A}\ker(f^*)$ is $0$. This proves the result.
    \end{proof}

    Hence, for $p,q\in S_{1}(M)=\beta A$, $\pi(p)=\pi(q)$ (in $bA = A^{*}/(A^{*})^{00}_{top}$) if $f^*(a) = f^{*}(b)$ for each $f\in\widehat A$ and realizations $a$ of $p$ and $b$ of $q$. In other words, if $f^{*}(p) = f^{*}(q)$ for all $f\in\widehat A$. 

    Hence $bA$ can again be seen as the space of $\mathcal F$-types, where $\mathcal F =\widehat A$. But now $f\in\mathcal F$ is considered as a definable over $M$ homomorphism from $G^*$ to $\mathbb T$ in the sense of first order logic (i.e. factoring through $S_{1}(M)$). In any case, its topology as a quotient of $\beta G$ and as a type space in continuous logic in the sense of Section 2 coincide.

    \subsection{Gelfand and Pontryagin dualities}

    In the functional analysis literature, \Cref{theorem:mainResult} occurs modulo Gelfand duality, but stated just for the case where $A$ is also locally compact.  For example, Theorem 4.79 in \cite{folland}, with some paraphrasing and some editing, reads:

    \begin{theorem*}
        Let $A$ be a locally compact  abelian  group and $f:A\to\mathbb C$ a bounded and continuous function. Then the following are equivalent:
        \begin{enumerate}
            \item $f$ factors through $\iota:A\to bA$.
            \item $f$ is the uniform limit of linear combinations of characters of $A$.
        \end{enumerate}
    \end{theorem*}

    \noindent Paraphrasing once more, this theorem says that $C(bA)$ and the C\textsuperscript*-algebra generated by the $\mathcal F$-formulas coincide. Taking their spectra yields, on the one hand, the Bohr compactification $bG$ of $G$ (see e.g. Theorem 1.16 in \cite{folland}); and on the other, the space of $\mathcal F$-types $S_\mathcal F(M)$. In Section 3 we go through the proof of this from a model-theoretic perspective, without invoking the nature of the space of types as a Gelfand dual.

    Alternatively, the literature also reveals the Pontryagin dual of a locally compact group $G$ as the spectrum of the Banach algebra $L^1(G)$ of Haar-integrable functions on $G$ modulo the Fourier transform  (\cite{folland} Theorem 4.2). Regarding the latter, we have
 
    \begin{proposition*}[\cite{folland} Proposition 4.13]
        The Fourier transform is a norm-decreasing *-homomorphism from $L^1(G)$ to $C_0(\widehat G)$ with dense image.
    \end{proposition*}

    \noindent In particular, for any locally compact abelian  group $A$, we have the Fourier transform $L^1((\widehat A)_d)\to C_0(\widehat{(\widehat A)_d})=C(bA)$. The C\textsuperscript*-algebra generated by $L^1((\widehat A)_d)$ is then the algebra of $\mathcal F$-formulas, as implied by the theorem above, and their spectra again are both the Bohr compactification $bA$ of $A$  and $S_\mathcal F(M)$.

\vspace{2mm}
\noindent
Returning to Theorem 4.79 in \cite{folland}, there is actually a third equivalent statement mentioned there: $f$ is uniformly almost periodic.  This is a strong form of stability, possibly connected to ``$1$-basedness" and we will investigate the implications in future work. 

In any case we have not managed to see our proofs in Section 3 ``mirroring"  the accounts just mentioned. 

    \bibliographystyle{plain}
    \bibliography{references}

\end{document}